\documentclass[12pt,a4paper,oneside]{amsart}

\usepackage{amsfonts, amsmath, amssymb, amsthm, amscd}
\usepackage{anysize}
\usepackage{graphicx}
\usepackage{pdfsync}
\usepackage{enumitem}

\usepackage{hyperref}
\hypersetup{
   colorlinks = true, 
   urlcolor = blue, 
   linkcolor = blue, 
   citecolor = red 
}

\newtheorem{theorem}{Theorem}[section]
\newtheorem{lemma}[theorem]{Lemma}

\newtheorem{conjecture}[theorem]{Conjecture}

\newtheorem{proposition}[theorem]{Proposition}

\theoremstyle{definition}
\newtheorem{definition}[theorem]{Definition}

\theoremstyle{remark}
\newtheorem{remark}[theorem]{Remark}
\newtheorem{question}[theorem]{Question}
\newtheorem{example}[theorem]{Example}

\numberwithin{equation}{section}

\DeclareMathOperator{\dist}{dist}

\DeclareMathOperator{\rint}{rint}
\DeclareMathOperator{\inte}{int}
\DeclareMathOperator{\conv}{conv}

\newcommand{\eps}{\varepsilon}

\begin{document}

\title{A multi-plank generalization of the Bang and Kadets inequalities}

%


\author[A.~Balitskiy]{Alexey~Balitskiy}

\email{balitski@mit.edu}

%

\address{Dept. of Mathematics, Massachusetts Institute of Technology, 182 Memorial Dr., Cambridge, MA 02142}
\address{Institute for Information Transmission Problems RAS, Bolshoy Karetny per. 19, Moscow, Russia 127994}

\thanks{Supported in part by the Russian Foundation for Basic Research Grant 18-01-00036}


\begin{abstract}
If a convex body $K \subset \mathbb{R}^n$ is covered by the union of convex bodies $C_1, \ldots, C_N$, multiple subadditivity questions can be asked. Two classical results regard the subadditivity of the width (the smallest distance between two parallel hyperplanes that sandwich $K$) and the inradius (the largest radius of a ball contained in $K$): the sum of the widths of the $C_i$ is at least the width of $K$ (this is the plank theorem of Th\o{}ger~Bang), and the sum of the inradii of the $C_i$ is at least the inradius of $K$ (this is due to Vladimir~Kadets).

We adapt the existing proofs of these results to prove a theorem on coverings by certain generalized non-convex ``multi-planks''. One corollary of this approach is a family of inequalities interpolating between Bang's theorem and Kadets's theorem. Other corollaries include results reminiscent of the Davenport--Alexander problem, such as the following: if an $m$-slice pizza cutter (that is, the union of $m$ equiangular rays in the plane with the same endpoint) in applied $N$ times to the unit disk, then there will be a piece of the partition of inradius at least $\frac{\sin \pi/m}{N + \sin \pi/m}$.
\end{abstract}

\maketitle

\section{Introduction}
\label{sec:intro}

Let $K$ be a convex set in $\mathbb{R}^n$ endowed with the Euclidean norm. Two basic quantities measuring the ``thickness'' of $K$ are its width $w(K)$, the smallest distance between two parallel hyperplanes that sandwich $K$, and its inradius $r(K)$, the largest radius of a ball contained in $K$. There are two classical results on the subadditivity of $w(\cdot)$ and $r(\cdot)$.

\begin{theorem}[Th.~Bang~{\cite{bang1951solution}}]
\label{thm:bang}
If a convex set $K$ is covered by convex sets $C_1, \ldots, C_N$, then
\[
\sum\limits_{i=1}^N w(C_i) \ge w(K).
\]
\end{theorem}

\begin{theorem}[V.~Kadets~{\cite{kadets2005coverings}}]
\label{thm:kadets}
If a convex set $K$ is covered by convex sets $C_1, \ldots, C_N$, then
\[
\sum\limits_{i=1}^N r(C_i) \ge r(K).
\]
\end{theorem}

If a convex set $K$ sits inside an affine subspace $L$ of $\mathbb{R}^n$, we use the notation $r(K;L)$ for the inradius of $K$ measured inside $L$.

\begin{definition}
\label{def:inrad}
Let $1 \le k \le n$. The following quantities will be called the \emph{intrinsic inradii} of a convex set $K \subset \mathbb{R}^n$.
\begin{enumerate}
  \item The \emph{upper intrinsic inradius} of $K$ is defined as
\[
r^{(k)}(K) = \inf\limits_{\dim L = k} r(K \vert L; L) = \inf\limits_{\dim L = k} r(K + L^\bot),
\]
where $L$ runs over the $k$-dimensional subspaces of $\mathbb{R}^n$, and $K \vert L$ is the orthogonal projection of $K$ onto $L$.

  \item The \emph{lower intrinsic inradius} of $K$ is defined as
\[
r_{(k)}(K) = \inf\limits_{\dim L = k} \sup\limits_{x \in L^\bot} r(K \cap (L+x); L+x),
\]
Equivalently, $r_{(k)}(K)$ can be defined via a Kakeya-type property: it is the supremum of numbers $r$ such that the open ball of radius $r$ of any $k$-dimensional subspace can be placed in $K$ after a translation.
\end{enumerate}
\end{definition}

Those radii appeared (under different names) in multiple papers, e.g.~\cite{betke1992estimating, betke1993generalization, brandenberg2005radii, henk2008intrinsic}. Some other notions of successive radii (different from ours) in the context of certain plank problems were considered in~\cite{bezdek1995solution, bezdek1996conway, bezdek2014plank}.

Observe that $r^{(1)}(K) = r_{(1)}(K) = w(K)/2$ and $r^{(n)}(K) = r_{(n)}(K) = r(K)$.
It is clear that $r^{(k)}(K) \ge r_{(k)}(K)$, but in general it might happen that this inequality is strict; for instance, this happens for the regular tetrahedron in $\mathbb{R}^3$ and $k=2$.


The following result, interpolating between Theorem~\ref{thm:bang} and Theorem~\ref{thm:kadets}, will follow as a corollary of the main theorem, Theorem~\ref{thm:pants}.
\begin{theorem}
\label{thm:subadditivity}
If a convex set $K$ is covered by convex sets $C_1, \ldots, C_N$, then for any $1 \le k \le n$,
\[
\sum\limits_{i=1}^N r^{(k)}(C_i) \ge r_{(k)}(K).
\]
\end{theorem}

Most often Theorem~\ref{thm:bang} is formulated in terms of coverings by planks; in this formulation it answered Tarski's question~\cite{tarski1932further}. A \emph{plank} is the set of all points between two parallel hyperplanes. We will interpret Theorem~\ref{thm:subadditivity} in terms of coverings by certain non-convex ``planks'' (see Definition~\ref{def:pants} and Figures~\ref{fig:simple}, \ref{fig:winged} for examples) and adapt classical proofs of Theorems~\ref{thm:bang},~\ref{thm:kadets} to give a one-page proof of a more general plank theorem (Theorem~\ref{thm:pants}).

Another type of corollaries that can be immediately deduced from the main theorem is akin to the Davenport--Alexander problem (see~\cite{alexander1968problem}, where its relation to Bang's plank problem is explained), and a version of Conway's fried potato problem (see~\cite{bezdek1995solution}, especially Theorem~2 therein). A corollary of Theorem~\ref{thm:pants} tells us that one can arbitrarily apply a commercial pizza cutter to one's favorite pizza several times and still find a decently-sized slice.

\begin{theorem}
\label{thm:potato}
Let us call an \emph{$m$-fan} ($m \ge 2$) the union of $m$ rays in the plane with the same endpoint and with all angles $\frac{2\pi}{m}$. If the unit disk is partitioned by $m$-fans $S_1, \ldots, S_N$, then there is a piece of inradius at least $\frac{\sin \pi/m}{N + \sin \pi/m}$.
\end{theorem}

In the case $m=2$ it recovers a well-known Davenport-type result, which is equivalent to Tarski's plank problem for disk.

Section~\ref{sec:plank} introduces our notion of a \emph{multi-plank} and gives several examples. The main theorem in Section~\ref{sec:pants} is followed by the proofs of Theorem~\ref{thm:subadditivity} and Theorem~\ref{thm:potato} (together with its higher-dimensional generalizations). Section~\ref{sec:voronoi} establishes further properties of multi-planks with the hope to illustrate the concept and make Definition~\ref{def:pants} less obscure.

Section~\ref{sec:normed} discusses to what extent the main theorem generalizes to the case when $\mathbb{R}^n$ is endowed with a non-Euclidean norm, whose unit ball need not be centrally symmetric, generally speaking. The normed counterparts of Theorems~\ref{thm:bang}~and~\ref{thm:kadets} are widely open questions. The former, known as Bang's conjecture on relative widths, is solved by K.~Ball~\cite{ball1991plank} for the case when the unit ball is centrally symmetric. The latter is far less understood, with some progress towards the case of partitions (instead of coverings) made in~\cite{akopyan2012kadets}.

\subsection*{Acknowledgements} The most general multi-plank definition was hinted to me by Roman Karasev, whom I thank sincerely. I am also grateful to Alexandr Polyanskii for the fruitful discussions that we had. I thank Aleksei Pakharev for his help with the figures. Finally, I thank the referee, whose comments improved the exposition.

\section{Multi-planks}
\label{sec:plank}

\begin{definition}
\label{def:pants}
Let $V = \{v^{1}, \ldots, v^{m}\}$, $m \ge 2$, be a set of points in $\mathbb{R}^n$, such that the closed ball with the smallest radius containing $V$ is centered at the origin. Denote by $r(V)$ its radius. 
\begin{enumerate}
  \item The set
  \[
  P = \left\{ x \in \mathbb{R}^n ~\middle\vert~ \forall j \in [m] ~ \exists j' \in [m] \text{ such that } |x| < |x - v^j + v^{j'}|  \right\}
  \]
  will be called the \emph{open centered multi-plank} generated by $V$.
  \item The closure $\overline{P}$ of $P$ will be called the \emph{closed centered multi-plank} generated by $V$.
  \item A \emph{multi-plank} generated by $V$ is a translate of $P$ or $\overline{P}$.
\end{enumerate}
In all these cases, the radius $r(V)$ will be called the \emph{inradius} of a multi-plank (this word choice will be justified by Lemma~\ref{lem:inradius}). The dimension of the affine hull of $V$ will be called the \emph{rank} of a multi-plank.
\end{definition}

\begin{example}
\label{ex:pants1}
If $V = \{u, -u\}$ for $0 \neq u \in \mathbb{R}^n$, then the corresponding rank $1$ (open centered) multi-plank is just the ordinary (open) plank
\[
P = \left\{x \in \mathbb{R}^n ~\middle\vert~ -|u|^2 < \langle x, u \rangle < |u|^2 \right\}.
\]
\end{example}

\begin{example}
\label{ex:pants2}
Let $V = \{v^1, \ldots, v^{n+1}\} \subset \mathbb{R}^n$ be a set of affinely independent vectors of length $r$ whose convex hull contains the origin in its interior. It follows that the smallest ball containing $V$ is $B_r$, the ball of radius $r$ centered at the origin. The corresponding rank $n$ (open centered) multi-plank $P$ can be described as follows. For each $j \in [n+1]$, draw the tangent hyperplane $H^j$ to the ball $B_r$ at the point $v^j$. Those hyperplanes bound a simplex $\Delta$. Consider the union $F$ of rays with the common endpoint at the origin that intersect the $(n-2)$-skeleton of $\Delta$. It is easy to check that $F$ is the fan dividing space into convex regions (in Section~\ref{sec:voronoi} it will be explained how these regions are related to the \emph{Vorono\u{\i} diagram} of $V$). The multi-plank $P$ looks like a thickened fan $F$, with the widths of its ``wings'' defined so that $\partial P$ passes through each of the $v^j$. (See Figure~\ref{fig:simple} for an example.)
\end{example}

\begin{figure}[ht]
\centering
\includegraphics[width=0.9\textwidth]{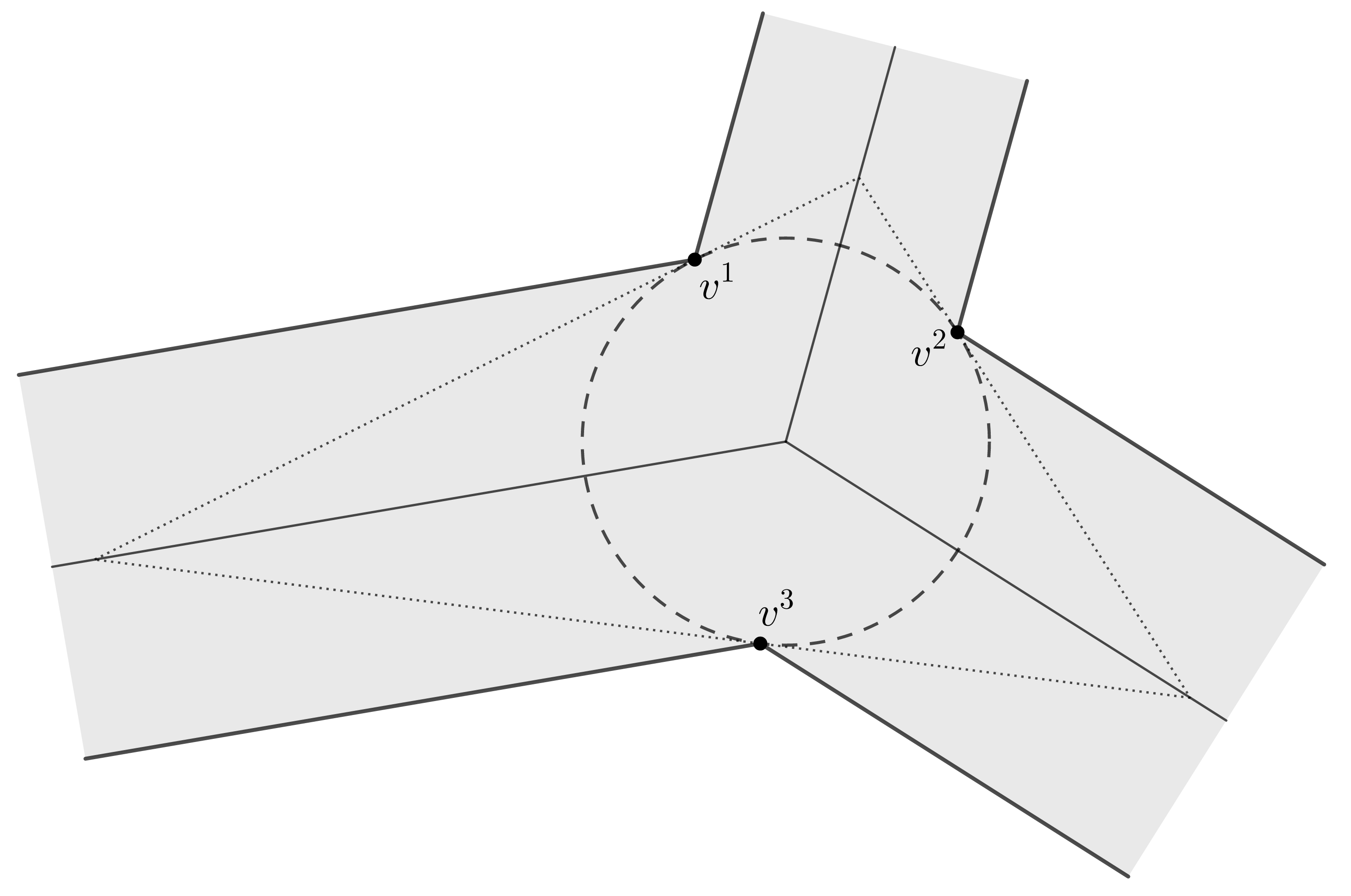}
\caption{A simple multi-plank in the plane}
\label{fig:simple}
\end{figure}

\begin{example}
\label{ex:pants3}
Let $V = \{v^1, \ldots, v^{k+1}\} \subset \mathbb{R}^n$, $1 \le k \le n$, be a set of vectors of length $r$ whose convex hull is $k$-dimensional and contains the origin in its relative interior. It follows that the smallest ball containing $V$ is $B_r$, the ball of radius $r$ centered at the origin. The corresponding rank $k$ (open centered) multi-plank $P$ is the Minkowski sum of the $k$-dimensional multi-plank generated by $V$ in its affine hull (as in the previous example) with the orthogonal $(n-k)$-dimensional subspace.
\end{example}

The multi-planks as in the examples above (and their closures) will be called \emph{simple}. An example of a non-simple multi-plank can be obtained, for instance, if one takes $V$ consisting of three vectors whose endpoints form an obtuse triangle: $v^1 = -v^2$ and $|v^3| < |v^1|$. Another important family of non-simple multi-planks shows up in Proposition~\ref{prop:coverplanks} (see Figure~\ref{fig:winged}).

\begin{lemma}
\label{lem:simple}
Let $1 \le k \le n$. Any open (closed) convex set $C$ with finite intrinsic radius $r^{(k)}(C)$ can be placed inside an open (closed) simple multi-plank of rank at most $k$ and inradius $r^{(k)}(C)$.
\end{lemma}

\begin{proof}
The quantity $r(C \vert L; L)$, where $L$ is a $k$-dimensional subspace, depends on $L$ in a lower semi-continuous way (in fact, it is continuous whenever finite). Therefore, we can pick $L$ delivering minimum in the definition $r^{(k)}(C) = \inf\limits_{\dim L = k} r(C \vert L; L)$. Let $c+B_r$ be the largest $k$-ball in $C \vert L \subset L$. Next, there are points $c + v^1, \ldots, c + v^{m} \in L$ (for some $2 \le m \le k+1$) in the intersection of (relative) boundaries of $C\vert L$ and of $c+B_r$, ``certifying'' that $r$ was indeed maximal, that is, the following two conditions are satisfied.
\begin{enumerate}[label=(C\arabic*)]
  \item\label{cond:1} The set $C$ lies in the intersection of halfspaces
\[
\left\{ x \in \mathbb{R}^n ~\middle\vert~ \langle x - c, v^j \rangle < r^2 \right\},
\]
over $j \in [m]$. (In the case when $C$ is closed, the inequalities are non-strict.)

  \item\label{cond:2} The point $c$ is contained in the relative interior of $\conv \{c + v^1, \ldots, c + v^{m}\}$.
\end{enumerate}

Let us show that the multi-plank $c+P$, where $P$ is generated by $V = \{v^1, \ldots, v^{m}\}$, as in Example~\ref{ex:pants3}, will do. By shifting everything, we can assume $c$ coincides with the origin. Suppose there exists a point $x \in C \setminus P$. By definition of the multi-plank $P$, there is $j \in [m]$ such that
\[
\left|x\right| \ge \left|x - v^j + v^{j'}\right| \quad \mbox{for all } v^{j'} \in V.
\]
(In the closed case, the inequalities are strict.)
Therefore
\begin{align*}
  \left|x\right|^2 &\ge \left|x - v^j + v^{j'}\right|^2 \\
  &= \left|x\right|^2 + 2r^2 - 2\left\langle x,v^j \right\rangle + 2\left\langle x-v^j,v^{j'} \right\rangle.
\end{align*}

By condition~\ref{cond:2} above it is possible to pick $v^{j'} \in V$ such that $\langle x - v^j, v^{j'} \rangle \ge 0$, hence having
\[
\left\langle x,v^j \right\rangle \ge r^2 + \left\langle x-v^j,v^{j'} \right\rangle \ge r^2,
\]
which contradicts condition~\ref{cond:1}.
\end{proof}

\section{Multi-plank theorem}
\label{sec:pants}

Now we are ready to state the main result. The proof follows closely the ideas of Bang and Kadets. Our exposition also makes use of a trick by Bogn\'ar~\cite{bognar1961w}.

Recall from Definition~\ref{def:pants} that if a multi-plank $P$ is generated by $V$, then $r(V)$ denotes the radius of the smallest closed ball containing $V$; we call it the \emph{inradius} of $P$, even though this word choice is not yet justified.

\begin{theorem}
\label{thm:pants}
If a convex set $K \subset \mathbb{R}^n$ is covered by multi-planks $P_1, \ldots, P_N$ of rank at most $k$ then
\[
\sum\limits_{i=1}^N r(V_i) \ge r_{(k)}(K),
\]
where $V_i$ is the generating set for $P_i$.
\end{theorem}

\begin{proof} Every closed multi-plank can be covered by an open one of almost the same inradius; so without loss of generality we assume that the multi-planks are open.

It suffices to consider the case when $K$ is bounded, i.e., $K \subset B_R$ for some $R$. If not, we apply theorem for $K \cap B_R$ and pass to the limit $R \to \infty$; here we use $r_{(k)}(K \cap B_R) \to r_{(k)}(K)$ as $R \to \infty$.

First we reduce the problem to the case of centered multi-planks and then deal separately with the centered case (this strategy can be traced back to Bogn\'ar~\cite{bognar1961w}).

\textbf{Step 1.} We think of $\mathbb{R}^n$ as a coordinate subspace $H \subset \mathbb{R}^{n+1}$; say, $H = \{(x, 0) \in \mathbb{R}^{n+1} ~\vert~ x \in \mathbb{R}^n\}$. Now both the set $K$ and the multi-planks $P_i$ sit inside $\mathbb{R}^{n+1}$. Pick a point $O = (0_n,D)$ very far from the origin; here $0_n$ is the origin of $\mathbb{R}^n$ and $D \in \mathbb{R}$ is large. For each $i$, build the cylinder $C_i = (P_i \cap B_R) + \ell_i$, where $\ell_i$ is the line passing through $O$ and through the center of $P_i$. Those cylinders cover the cone $\widehat K = \conv(K \cup \{O\})$. There are two statements to check:
\begin{enumerate}
  \item each $C_i$ can be covered by a multi-plank of the same rank as $P_i$, centered at $O$, and of inradius close to $r(V_i)$ (the proximity depends on $D$);
  \item the intrinsic inradius $r_{(k)}\left(\widehat K\right)$ is close to $r_{(k)}(K; H)$. (The notation $r_{(k)}(\cdot; H)$ is to specify the ambient space where the intrinsic inradius is measured.)
\end{enumerate}

For the first one, notice that $C_i$ splits as the orthogonal product of $\ell_i$ and of the set $A(D)$ which is an affine copy of $P_i \cap B_R$ shrunk negligibly (as long as $D$ is large) along one direction. The reader can convince themselves that $A(D)$ can be covered by a scaled copy of $P_i \cap B_R$ with the homothety coefficient tending to $1$ as $D \to \infty$. This explains the first statement.

For the second claim, fix $d \in \mathbb{R}$ large enough so that $r_{(k)}(K; H) = r_{(k)}(K+[0,d])$; here $K+[0,d]$ is a shorthand for $K + [(0_n,0),(0_n,d)] \subset \mathbb{R}^{n+1}$. Now observe that $\widehat K \cap (K+[0,d])$ converges to $K+[0,d]$ in the Hausdorff metric, as $D \to \infty$. One can check that the function $r_{(k)}(\cdot)$ is Hausdorff continuous, so we can write
\[
r_{(k)}(K; H) \ge r_{(k)}\left(\widehat K\right) \ge r_{(k)}\left(\widehat K \cap (K+[0,d])\right) \underset{D \to \infty}\longrightarrow r_{(k)}(K+[0,d]) = r_{(k)}(K; H).
\]

Now, applying the theorem in the centered case, we get the desired inequality with a small error term, which decays as $D \to \infty$.

\textbf{Step 2.}
Now we can assume that all the $P_i$ are centered at the origin.
The proof here follows closely the ideas from original papers by Bang and Kadets with certain simplifications.
Assume the contrary to the statement of theorem:
\[
\alpha = \frac{r_{(k)}(K)}{\sum\limits_{i=1}^N r(V_i)} > 1.
\]
We define the \emph{Bang set}
\[
X = \left\{ \sum_{i=1}^{N} v_i^{j_i} ~\middle\vert~ 1 \le j_i \le m_i \right\},
\]
where $V_i = \{v_i^{1}, \ldots, v_i^{m_i}\}$ is the generating set of $P_i$. The strategy of the proof is to show that $X$ can be covered by a translate of $K$ (assuming the contrary to the statement of theorem) but at the same time $X$ does not fit into $\bigcup P_i$.

\textit{Step 2.1.} The Bang set splits as the Minkowski sum of the generating sets of the multi-planks: $X = V_1 + \ldots + V_N$. By the definition of the lower intrinsic radius, $V_i$ can be covered by a translate of $\frac{r(V_i)}{r_{(k)}(K)} \overline{K}$ (where bar denotes closure), hence by a translate of $\frac{\alpha r(V_i)}{r_{(k)}(K)} K$. Therefore, for some translation vector $s \in \mathbb{R}^n$,
\[
X = V_1 + \ldots + V_N \subset s + \frac{\alpha r(V_1)}{r_{(k)}(K)} K + \ldots + \frac{\alpha r(V_N)}{r_{(k)}(K)} K = s + K.
\]

\textit{Step 2.2.} Suppose $X \subset s + \bigcup P_i$, $s \in \mathbb{R}^n$. Consider the farthest from the origin point in $X-s$; let it be $x = -s + \sum_{i=1}^{N} v_i^{j_i}$. Fix $i$ and consider the family of vectors $(x - v_i^{j_i}) + v_i^{j_i'}$, over $j_i' \in [m_i]$. Since $x$ is the longest among them, it follows that $x \notin P_i$. Repeating this over all $i$, we get a contradiction.
\end{proof}

\begin{proof}[Proof of Theorem~\ref{thm:subadditivity}]
If a convex set $K$ is covered by convex sets $C_1, \ldots, C_N$, then each $C_i$ can be covered by a simple closed multi-plank with inradius $r^{(k)}(C_i)$ (see Lemma~\ref{lem:simple}). Now Theorem~\ref{thm:pants} implies the desired inequality:
\[
\sum\limits_{i=1}^N r^{(k)}(C_i)  \ge r_{(k)}(K).
\]
\end{proof}

\begin{proof}[Proof of Theorem~\ref{thm:potato}]
Let $\alpha_F = \frac{\pi}{m}$. Suppose the contrary, and denote $r < \frac{\sin \alpha_F}{N + \sin \alpha_F}$ the radius of the largest disk inscribed in the partition by fans. Pick a number $\overline r$ between $r$ and $\frac{\sin \alpha_F}{N + \sin \alpha_F}$. Then the disk $B_{1-r}$ is covered by the multi-planks $P_1, \ldots, P_N$, where $P_i$ is the $\overline r$-neighborhood of $S_i$. The inradius of each multi-plank equals $\frac{\overline r}{\sin \alpha_F}$, so using Theorem~\ref{thm:pants} one gets the inequality
\[
\frac{N \overline r}{\sin \alpha_F} \ge 1-r > \frac{N}{N + \sin \alpha_F},
\]
contradicting the assumption $\overline r < \frac{\sin \alpha_F}{N + \sin \alpha_F}$.
\end{proof}

Theorem~\ref{thm:potato} can be generalized to higher dimensions in the evident way; the only difficulty is to write down the guaranteed inradius in terms of the class of ``pizza cutters''. In the examples below, the cutter shape is given by a certain fan $F$, dividing $\mathbb{R}^n$ into convex cones so that $F$ cuts out in the unit sphere $S^{n-1}$ a bunch of regions all having the same inradius $\alpha_F$ in the intrinsic sphere metric $\dist_{S^{n-1}}(\cdot,\cdot)$.
\begin{enumerate}
  \item One can define an $m$-fan $F$ in $\mathbb{R}^n$ as the union of $m$ half-planes of dimension $(n-1)$ sharing the same boundary $(n-2)$-subspace, and with the dihedral angles all equal to $2 \alpha_F = \frac{2\pi}{m}$.
  \item For every regular polytope $C \subset \mathbb{R}^n$ centered at the origin, one can consider the fan $F$ consisting of the rays from the origin passing through the $(n-2)$-skeleton of $C$. The regions cut out by $F$ in the unit sphere are all congruent since $C$ is regular. For example, in the case of regular simplex, $\alpha_F = \arccos \frac{1}{n}$.
  \item For every Coxeter hyperplane arrangement $\mathcal{A}$ in $\mathbb{R}^n$ (that is, the set of hyperplanes passing through the origin and generating a finite reflection group), one can consider the fan $F$ consisting of the hyperplanes of $\mathcal{A}$. The regions cut out by $F$ in the unit sphere are all congruent since the reflection group acts transitively on the Weyl chambers. For example, in the case of type $A_n$ reflection group, $\alpha_F = \arccos \sqrt{\frac{3}{2(n-1)n(n+1)}}$.
\end{enumerate}

All these examples are subsumed by the following more general Davenport-type theorem, whose proof is literally the same as the one of Theorem~\ref{thm:potato}.

\begin{theorem}
\label{thm:multipotato}
Let $G$ be a finite subgroup of $\mbox{SO}(n)$, acting on the unit sphere, and let $O$ be the $G$-orbit of any point from the unit sphere. The Vorono\u{\i} diagram of the set $O$ gives rise to a fan $F \subset \mathbb{R}^n$ as follows: by definition, $x \in F$ if the function $f(y) = |x-y|, y \in O$, attains its minimum in at least two orbit points. The regions cut out by $F$ in the unit sphere are all congruent since $G$ acts transitively on them, and their spherical inradius is
\[
\alpha_F = \frac12 \min\limits_{y_1 \neq y_2 \in O} \dist_{S^{n-1}}(y_1,y_2).
\]
Now, if one cuts the unit ball of $\mathbb{R}^n$ by $N$ congruent copies of $F$, there will be a piece of inradius at least $\frac{\sin \alpha_F}{N + \sin \alpha_F}$.
\end{theorem}

\begin{question}
Let $F$ be a sufficiently regular codimension 2 ``fan''. For instance, $F$ can be the union of the rays from the origin passing through the $(n-3)$-skeleton of the regular $n$-simplex, $n \ge 3$. The cuts are given by $N$ congruent copies of $F$ placed arbitrarily in $\mathbb{R}^n$. What is the largest radius of an open ball lying in the unit ball and avoiding the cuts?
\end{question}

We finish this section with a brief discussion of the optimality of the main theorem. Theorem~\ref{thm:pants} has many ``asymptotic equality cases'', different from trivial equality cases when $N=1$ or $k=1$. For example, the unit disc in the plane can be covered by two multi-planks of radius $r$ slightly greater than $1/2$, each generated by $N \gg 1$ points equidistributed along the circle of radius $r$. By picking $N$ sufficiently large one can get $r$ arbitrarily close to $1/2$. A similar example shows that Theorem~\ref{thm:potato} is asymptotically sharp for each fixed $m$ and $N \to \infty$. On the other hand, all those non-trivial asymptotic equality cases involve multi-planks that are not simple. Meanwhile, the proof of Theorem~\ref{thm:subadditivity} only exploits simple multi-planks. Given that, it would be interesting to know how sharp Theorem~\ref{thm:subadditivity} is when, say, $K$ is not centrally symmetric and $N > 1$.

\section{Multi-plank stratification}
\label{sec:voronoi}

This section is devoted to a complete description of how multi-planks actually look like. First, we show how multi-planks can be efficiently used for covering unions of conventional planks. This idea might be helpful for certain plank problems. Next, we introduce the language of \emph{anti-Vorono\u{\i} diagrams}, and their dual \emph{anti-Delaunay triangulations}. Recall Example~\ref{ex:pants2}: a simple multi-plank looks like an inflated fan, dividing $\mathbb{R}^n$ into unbounded convex regions, which form the so-called \emph{Vorono\u{\i} diagram} of $V$. An equivalent definition of a multi-plank, introduced in this section, tells us that this is always the case: associated to $V$, there is a nice subdivision of $\mathbb{R}^n$ into unbounded convex regions, such that its separating set can be thickened in order to get the multi-plank generated by $V$. The main result of this section, Theorem~\ref{thm:delaunay}, roughly speaking, explains how this thickening is done. We use it to justify the word ``inradius'' used in Definition~\ref{def:pants}. Remarks~\ref{rem:bezdek}~and~\ref{rem:polyanskii} discuss potential applications of the techniques of this section.

\begin{proposition}
\label{prop:coverplanks}
Let $V$ be the \emph{Bang set} of the family of planks
\[
P_i = \left\{x \in \mathbb{R}^n ~\middle\vert~ -|u_i|^2 < \langle x, u_i \rangle < |u_i|^2 \right\};
\]
that is, $V$ consists of all combinations $\sum\limits_i \pm u_i$, over all possible sign choices. Then the open centered multi-plank $P$ generated by $V$ contains the union $\bigcup\limits_i P_i$ (see Figure~\ref{fig:winged}).
\end{proposition}

\begin{proof}
Assume $x \notin P$; then for a certain $v^j = \sum\limits_i \eps_i u_i$, $\eps_i \in \{+1, -1\}$, we have
\[
\left|x\right| \ge \left|(x - v^j) + \sum_i \eps_i' u_i\right|, \quad \mbox{for all } \eps_i' \in \{+1, -1\}.
\]
Therefore,
\[
\left|x\right|^2 \ge \left|x + \sum_i (\eps_i' - \eps_i) u_i\right|^2, \quad \mbox{for all } \eps_i' \in \{+1, -1\}.
\]
Set all $\eps_i'$ equal to the corresponding $\eps_i$ except one; then we get
\[
\left|x\right|^2 \ge \left|x - 2\eps_i u_i\right|^2 = \left|x\right|^2 - 4\eps_i \langle x, u_i \rangle + 4|u_i|^2, \quad \mbox{for all } i.
\]
This last line implies that $x \notin P_i$, for each $i$; hence, $\bigcup\limits_i P_i \subseteq P$.
\end{proof}

\begin{figure}[ht]
\centering
\includegraphics[width=0.75\textwidth]{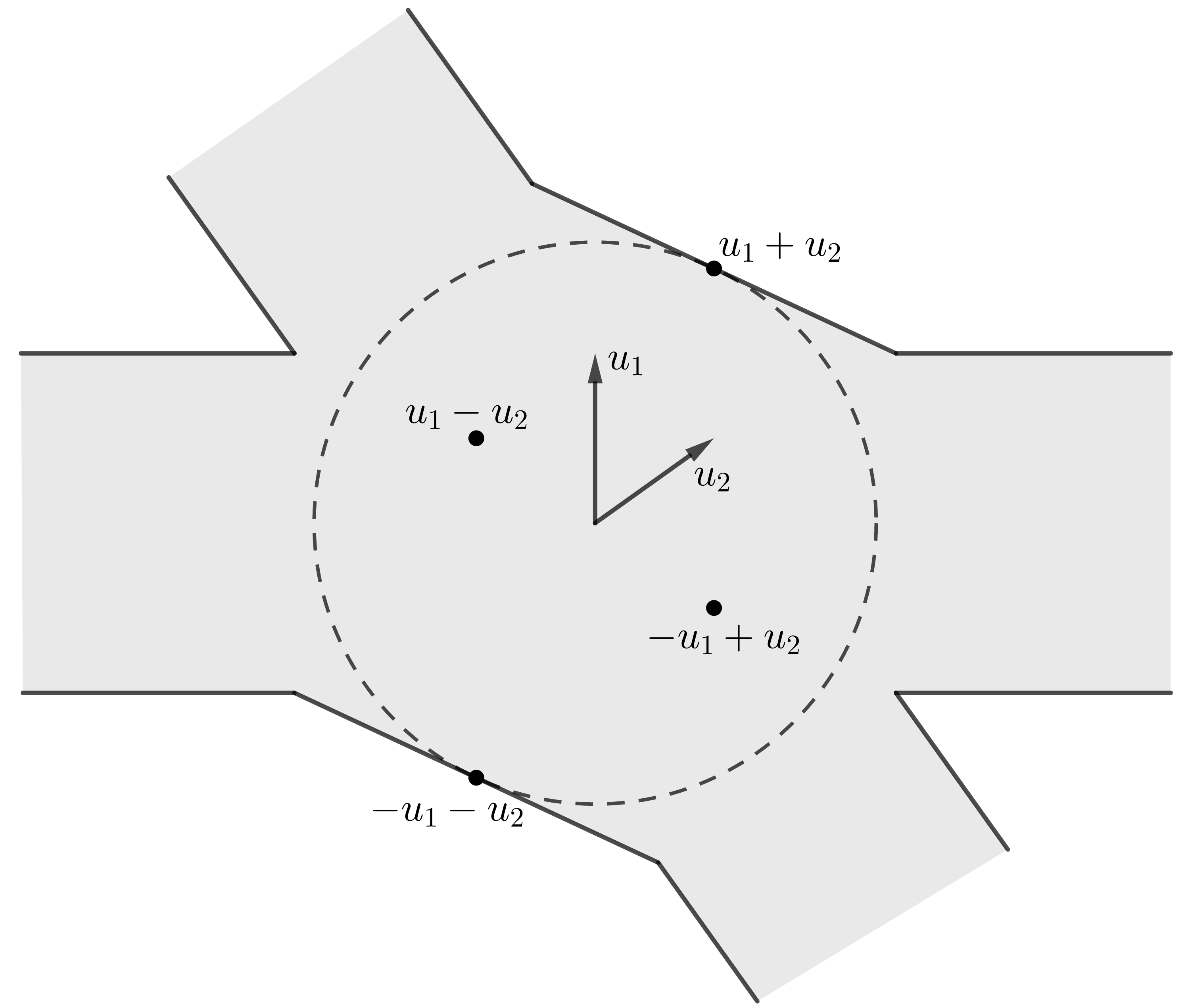}
\caption{A multi-plank covering two planks}
\label{fig:winged}
\end{figure}

\begin{remark}
\label{rem:bezdek}
Notice that two non-parallel planks, as in Figure~\ref{fig:winged}, have total half-width $|u_1| + |u_2|$, but their union is covered (even strictly covered unless $\langle u_1, u_2 \rangle = 0$) by a multi-plank of inradius $\max(|u_1+u_2|, |u_1-u_2|)$, which is strictly less than $|u_1| + |u_2|$. Informally speaking, this gives a more ``economical'' way of measuring ``total width'' in the following sense. Assume a symmetric convex set $K$ (for instance, the unit ball) is covered by planks; then some non-parallel pairs (or concurrent triples, etc.) are replaced by larger multi-planks, as in Proposition~\ref{prop:coverplanks}. For the resulting covering, the inequality given by Theorem~\ref{thm:pants} is stronger than the one given by Theorem~\ref{thm:bang} for the original covering. It is interesting whether this approach can give anything non-trivial for any of the open plank problems, e.g., Andr{\'a}s Bezdek's conjecture~\cite{bezdek2003covering} on covering an annulus: if the unit disk with a small puncture is covered by planks, then their total width is at least $2$.
\end{remark}

Now we introduce an equivalent way to define a multi-plank using the language of Vorono\u{\i} diagrams.

\begin{definition}
\label{def:voronoi}
Given a set $V = \{v^{1}, \ldots, v^{m}\}$ of points in $\mathbb{R}^n$, the \emph{anti-Vorono\u{\i} diagram} (or the \emph{farthest-point Vorono\u{\i} diagram}) is the partition $\mathbb{R}^n = \bigcup\limits_{j \in [m]} A_V^{j}$, where the closed \emph{cells} $A_V^{1}, \ldots, A_V^{m}$ are given by
\[
A_V^{j} = \left\{x \in \mathbb{R}^n ~\middle\vert~ |x-v^{j}| \ge |x-v^{j'}| ~\forall j' \in [m]\right\}.
\]
In other words, $A_V^{j}$ consists of all points for which the farthest element of $V$ is $v^{j}$.
\end{definition}

One should notice that all regions $A_V^{j}$ are convex. Additionally, each cell $A_V^{j}$ is either unbounded (if $v^{j}$ is an extreme point of $\conv V$) or empty (otherwise). The unboundedness of non-empty cells follows from the following simple claim: If $x \in A_V^{j}$ then the entire ray $\{x + t(x-v^{j}) ~\vert~ t \ge 0\}$ lies in $A_V^{j}$.
In the sequel we use the notation $A_{-V}^{j}$ for the anti-Vorono\u{\i} cell of the set $-V = \{-v^{1}, \ldots, -v^{m}\}$ corresponding to the point $-v^j$; that is,
\[
A_{-V}^{j} = \left\{x \in \mathbb{R}^n ~\middle\vert~ |x+v^{j}| \ge |x+v^{j'}| ~\forall j' \in [m]\right\}.
\]

We are ready to rephrase Definition~\ref{def:pants}. Let $V = \{v^{1}, \ldots, v^{m}\}$, $m \ge 2$, be a set of points in $\mathbb{R}^n$, such that the smallest ball containing $V$ is centered at the origin. Then the set
  \[
  P = \mathbb{R}^n \setminus \bigcup\limits_{j \in [m]} \left( v^{j} + A_{-V}^{j} \right)
  \]
is precisely the open centered multi-plank generated by $V$.

\begin{definition}
\label{def:delaunay}
Given a set $V = \{v^{1}, \ldots, v^{m}\}$ of points in $\mathbb{R}^n$, whose affine hull is the entire $\mathbb{R}^n$, an \emph{anti-Delaunay triangulation} (or a \emph{farthest-point Delaunay triangulation}) is a triangulation of $\conv V$ satisfying the \emph{full sphere property}: for each simplex of the triangulation, the (closed) ball whose boundary passes through the simplex vertices contains the entire $V$.
\end{definition}

It is known that an anti-Delaunay triangulation always exists (see, e.g.,~\cite[Section~4]{akopyan2011extremal}), and is unique provided that no $n+2$ points lie on a sphere. In the case when the affine hull of $V$ is smaller than $\mathbb{R}^n$, one can define an anti-Delaunay triangulation inside the affine hull of $V$.

Now we give a finer description what a multi-plank looks like. We use the notation $N_T(x)$ for the cone of outer normals of a convex body $T$ at a boundary point $x \in \partial T$; by definition, $N_T(x) = \{\nu \in \mathbb{R}^n ~\vert~ \langle \nu, y-x \rangle \le 0 ~\forall y \in T\}$. If $F$ is a face of $T$, we write $N_T(F)$ for the cone of outer normals at any point from the relative interior of $F$.

Let $V = \{v^{1}, \ldots, v^{m}\} \subset \mathbb{R}^n$ be the generating set of an open centered multi-plank $P$. If the rank of $P$ is smaller than $n$, the multi-plank looks like the orthogonal product of a subspace and a lower-dimensional multi-plank. For this reason, we restrict our attention to full-rank multi-planks for now.

Consider an anti-Delaunay triangulation $\Sigma$ of $\conv V$ regarded as a simplicial complex. For each top-dimensional cell $\sigma$ of $\Sigma$, let $S_\sigma$ be the translated copy of $\sigma$ such that the origin is equidistant from the vertices of $S_\sigma$. The simplices $S_\sigma$ do not overlap: this follows from the full sphere property of $\Sigma$. Indeed, if $\sigma_1$ and $\sigma_2$ are two anti-Delaunay cells, they need to be pushed apart in order to make their circumspheres concentric. (Here and below by ``circumsphere'' we mean the sphere passing through all the vertices of a simplex, and ``circumradius'' refers to its radius. Note that this is not standard.)


Let $\tau$ be a cell in $\Sigma$ of dimension greater that $0$. For each top-dimensional cell $\sigma$ containing $\tau$, find the corresponding face $T_{\tau, \sigma}$ of $S_\sigma$ (the one that is a translated copy of $\tau$). In particular, $T_{\sigma, \sigma} = S_\sigma$. Consider the following set:
\begin{equation}\label{eq:stratum}
P_\tau = \bigcap_{\sigma \supset \tau} \left( \rint T_{\tau, \sigma} + N_{S_{\sigma}}(T_{\tau, \sigma}) \right), \tag{$\star$}
\end{equation}
where the intersection is taken over all top-dimensional cells $\sigma$ containing $\tau$. For top-dimensional cells, this definition gives
\[
P_\sigma = \inte S_\sigma.
\]
We prove that the multi-plank $P$ can be decomposed into strata $P_\tau$.
\begin{theorem}
\label{thm:delaunay}
Let $P$ be a multi-plank of full rank generated by $V \subset \mathbb{R}^n$. Let $\Sigma$ be an anti-Delaunay triangulation of $V$, viewed as a simplicial complex. For each top-dimensional cell $\sigma$ of $\Sigma$, let $S_\sigma$ be the translated copy of $\sigma$ such that the origin is equidistant from the vertices of $S_\sigma$. For each pair of cells $\tau \subset \sigma$ with $0 < \dim \tau \le \dim \sigma = n$, let $T_{\tau, \sigma}$ be the face of $S_\sigma$ that is a translated copy of $\tau$. Let $P_\tau$ be defined as in~\eqref{eq:stratum}.

With this notation, the multi-plank $P$ admits the following stratification:
\[
P = \bigcup_{d=1}^n \bigcup_{\dim \tau = d} P_\tau,
\]
where the inner union is taken over all cells of $\Sigma$ of dimension $d$.
\end{theorem}

Note that we do not assume any genericity of $V$, except for its affine rank; but in the special case when no $n+2$ points of $V$ lie on the same sphere, this description tells us that locally, near each simplex $S_\sigma = \conv \{u^0, \ldots, u^n\}$, the multi-plank $P$ looks like $\mathbb{R}^n \setminus \bigcup\limits_{j = 0}^n (u^j + N_{S_\sigma}(u^j))$ (see Figure~\ref{fig:strata}).

\begin{figure}[ht]
\centering
\includegraphics[width=0.85\textwidth]{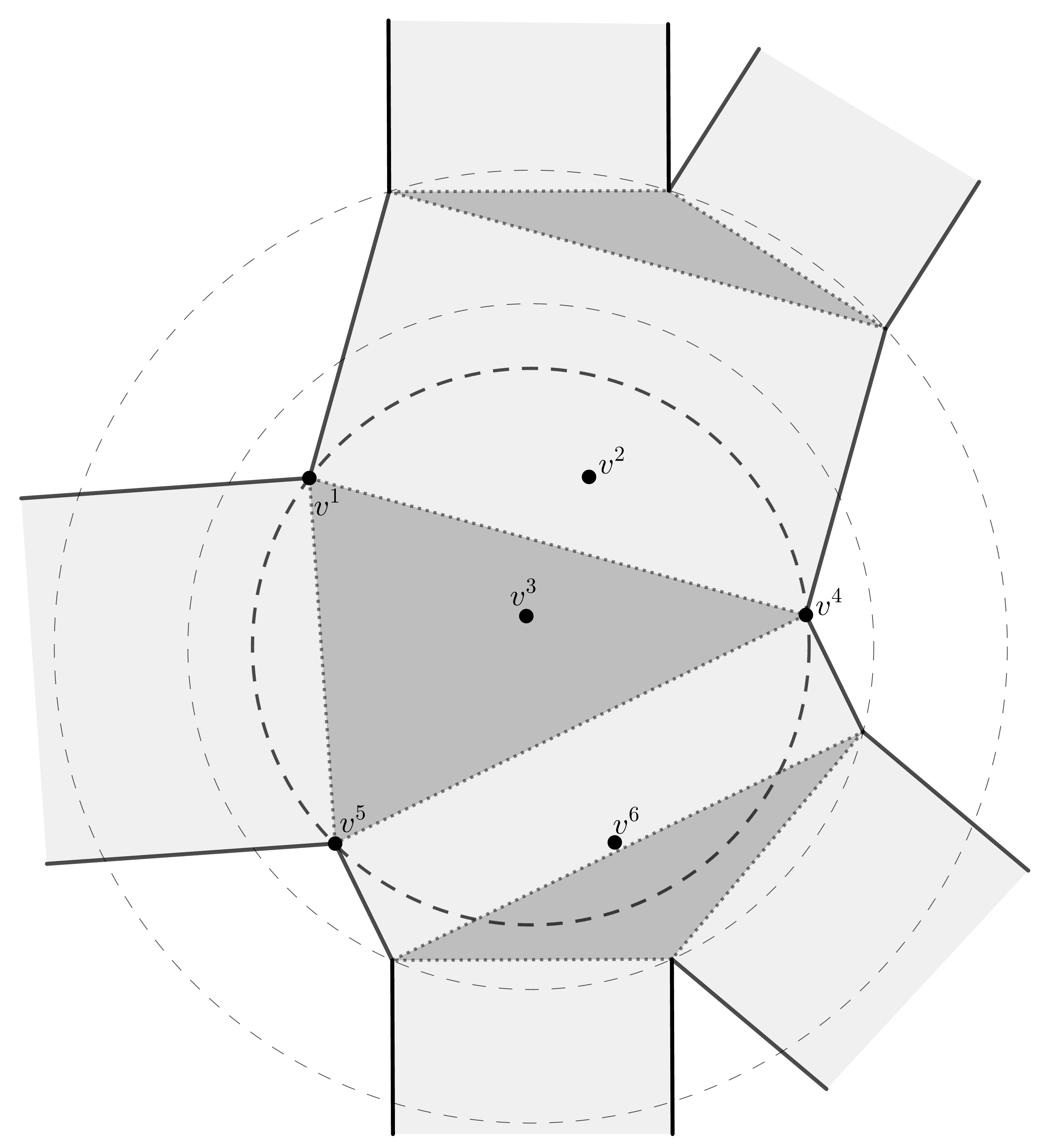}
\caption{The stratification of a multi-plank}
\label{fig:strata}
\end{figure}

\begin{proof}[Proof of Theorem~\ref{thm:delaunay}]
To begin with, we extend the definition of $P_\tau$ to the case when $\tau = v^j$ is a vertex in $\Sigma$. As before, for each top-dimensional cell $\sigma$ containing $v^j$, find the corresponding vertex $T_{v^j, \sigma}$ of $S_\sigma$. The stratum corresponding to $v^j$ is defined as
\[
P_{v^j} = \bigcap_{\sigma \ni v^j} \left( T_{v^j, \sigma} + N_{S_{\sigma}}(T_{v^j, \sigma}) \right),
\]
where the intersection is taken over all top-dimensional cells $\sigma$ containing $v^j$. We claim that $P_{v^j}$ is nothing else as $v^j + A_{-V}^j$, the shifted anti-Vorono\u{\i} cell from the alternative definition of a multi-plank.
This is somewhat tedious but straightforward. The set $v^j + A_{-V}^j$ is defined by the system of inequalities $|x| \ge |x-v^j+v^{j'}|$, over $v^{j'} \in V$. In fact, only the vertices $v^{j'}$ adjacent to $v^j$ in $\Sigma$ contribute to this system; this is essentially the duality between the anti-Delaunay triangulation and the anti-Vorono\u{\i} diagram. Equivalently, one can write those inequalities as

\begin{equation}\label{eq:strata1}
\left\langle x + \frac{v^{j'} - v^j}{2}, v^{j'} - v^j \right\rangle \le 0, \quad v^{j'} \mbox{ adjacent to } v^j. \tag{$\spadesuit$}
\end{equation}

On the other hand, each set $T_{v^j, \sigma} + N_{S_{\sigma}}(T_{v^j, \sigma})$ is defined by inequalities of the form
\begin{equation}\label{eq:strata2}
\left\langle x - T_{v^j, \sigma}, T_{v^{j'}, \sigma} - T_{v^j, \sigma} \right\rangle \le 0, \quad v^{j'} \in \sigma. \tag{$\heartsuit$}
\end{equation}

Varying $\sigma \ni v^j$ here, one gets those inequalities for all $v^{j'}$ adjacent to $v^j$ in $\Sigma$.

Observe that $v^{j'} - v^j = T_{v^{j'}, \sigma} - T_{v^j, \sigma}$, if $v^j$ and $v^{j'}$ form an edge in $\sigma$. Next,
\[
T_{v^j, \sigma} = \underbrace{\frac{T_{v^j, \sigma} + T_{v^{j'}, \sigma}}{2}}_{\text{orthogonal to } v^{j'} - v^j} + \underbrace{\frac{T_{v^j, \sigma} - T_{v^{j'}, \sigma}}{2}}_{= \frac{v^{j} - v^{j'}}{2}}.
\]
Therefore, $\frac{v^{j} - v^{j'}}{2}$ and $T_{v^j, \sigma}$ lie in the same hyperplane orthogonal to $v^{j'} - v^j$; this proves that inequalities~\eqref{eq:strata1}~and~\eqref{eq:strata2} are equivalent, so $P_{v^j} = v^j + A_{-V}^j$.

To finish the proof, it suffices to observe that $\mathbb{R}^n$ is the disjoint union of the strata $P_\tau$, over faces $\tau$ in $\Sigma$ of any dimension. Indeed, for any point $x \in \mathbb{R}^n$ we can consider the nearest to $x$ point $y \in \bigcup\limits_{\sigma} S_\sigma$ (the union is over top-dimensional cells of $\Sigma$). If $y \in \rint T(\tau, \sigma)$ then it is easy to see that $x \in P_\tau$ (and that $P_\tau$ is the only stratum containing $x$).
\end{proof}

\begin{remark}
\label{rem:delaunay}
Theorem~\ref{thm:delaunay} remains true in the case when the rank of $P \subset \mathbb{R}^n$ is less than $n$. In this case, the anti-Delaunay triangulation of $V$ should be considered inside the affine hull $L$ of $V$, and in the definition of stratum~\eqref{eq:stratum} the normal cone $N_{S_{\sigma}}(T_{\tau, \sigma})$ gets decomposed as the Minkowski sum of the normal cone in $L$ with $L^\bot$.
\end{remark}

\begin{remark}
\label{rem:polyanskii}
If $P$ is a centered rank $k$ multi-plank generated by $V$, the stratification of $P$ is defined using the $k$-dimensional anti-Delaunay triangulation of $V$. Let $\rho$ be the smallest circumradius of a top-dimensional cell of that triangulation. Clearly, $\rho \ge r(V)$, and it might happen that $\rho > r(V)$, if not all vertices of $\conv V$ lie on the sphere of radius $r(V)$. A direct corollary of Theorem~\ref{thm:delaunay} is that inside the ball $B_\rho$ of radius $\rho$ the multi-plank $P$ can be simplified; namely, $P \cap B_\rho = P' \cap B_\rho$, where the multi-plank $P'$ is generated by the subset of $V$ consisting of vectors of length $r(V)$. It would be interesting to know whether the proof of Jiang and Polyanskii~\cite{jiang2017proof} of L.~Fejes~T\'oth's zone conjecture can be retold using this trick in the language of multi-planks.\footnote{While this paper was under review, Polyanskii released a preprint~\cite{polyanskii2020cap} proving even stronger version of L.~Fejes~T\'oth's zone conjecture. The proof is partially inspired by certain multi-plank-related intuition, but there are other crucial ideas as well.}
\end{remark}

We use Theorem~\ref{thm:delaunay} to justify the word ``inradius'' used in Definition~\ref{def:pants}. Notice that here we refer to the intrinsic inradii of a possibly non-convex set; these are defined exactly as in Definition~\ref{def:inrad}.

\begin{lemma}
\label{lem:inradius}
Let $P$ be an open multi-plank of rank $k$ generated by $V \subset \mathbb{R}^n$. The radius $r(V)$ (as in Definition~\ref{def:pants}) is indeed the inradius of $P$;
moreover, the upper intrinsic inradii $r^{(k)}(P), \ldots, r^{(n)}(P)$, and the lower intrinsic radii $r_{(k)}(P), \ldots, r_{(n)}(P)$ all equal $r(V)$.
\end{lemma}
\begin{proof}
We can assume that $P$ is centered. First we show that the open ball $B_r$ of radius $r = r(V)$ is contained in $P$.

Suppose $x \notin P$, then $x \in v^{j} + A_{-V}^{j}$ for some $j$. It means that the farthest from $x - v^j$ element of $-V$ is $-v^j$, that is,
\[
\left|(x - v^j) + v^j\right| \ge \left|(x - v^j) + v^{j'} \right|, \quad \mbox{for all } v^{j'} \in V.
\]
Therefore,
\[
\left|x\right|^2 \ge \left|x - (v^j - v^{j'}) \right|^2 = \left| x - v^j \right|^2 + 2\left\langle x - v^j, v^{j'} \right\rangle + \left|v^{j'}\right|^2.
\]
It is possible to pick $v^{j'} \in \partial B_r$ such that $\langle x - v^j, v^{j'} \rangle \ge 0$, since $\overline B_r$ is the smallest ball containing $V$. For such a choice of $v^{j'}$ one gets
\[
\left|x\right|^2 \ge \left|v^{j'}\right|^2 = r^2,
\]
thus proving that $x \notin B_r$.

We have shown that $r_{(n)}(P) \ge r(V)$. Now we need to show that $r^{(k)}(P) \le r(V)$. In fact, it suffices to show that $r^{(n)}(P) \le r(V)$, since a rank $k$ multi-plank is the Minkowski sum of a $k$-dimensional multi-plank with the orthogonal subspace.

Let $c + B_\rho$ be the largest open ball contained in $P$, $\rho = r^{(n)}(P)$. Let its center belong to the stratum $P_\tau$ of $P$, where $\tau = \conv \{u^0, \ldots, u^d\}$ is a $d$-dimensional cell of the anti-Delaunay triangulation of $\conv V$. The stratum $P_\tau$ can be represented as $P_\tau = s + \rint \tau + R$, where $s \in \mathbb{R}^n$ is a translation vector, and $R$ is a certain closed set of dimension $n - d$, orthogonal to $\tau$. From the stratification result, Theorem~\ref{thm:delaunay}, one can deduce that the sets $s + u^j + R$ are all disjoint from $P$. Hence, the radius $\rho$ does not exceed the shortest among the distances $\dist(c, s + u^j + R) = |\pi_\tau(c-s) - u_j|$, where $\pi_\tau(c-s) \in \rint \tau$ is the orthogonal projection of $c-s$ onto the affine hull of $\tau$. If $\rho > r = r(V)$, then the vertices of $\tau$ are all in $\overline B_r \setminus (\pi_\tau(c-s) + B_{\rho})$, which can be strictly separated from $\pi_\tau(c-s)$ by a hyperplane; this contradicts the fact $\pi_\tau(c-s) \in \rint \tau$. Therefore, $\rho = r^{(n)}(P) \le r(V)$.
\end{proof}

\section{Multi-planks in normed spaces}
\label{sec:normed}

The scheme of the proof of Theorem~\ref{thm:pants} can be repeated to an extent in the setting of a normed space. Let $\mathbb{R}^n$ be endowed with a (possibly, asymmetric) norm $\|\cdot\|$ whose open unit ball is $B$, an open bounded convex set containing the origin:
\[
\|x\| = \inf \{r ~\vert~ x \in r B\}.
\]
We do not require $B$ to be centrally symmetric, so in general $\|x\| \neq \lVert -x \rVert$ (but the triangle inequality holds).

\begin{definition}
\label{def:normedinrad}
Let $K$ be a convex set in an asymmetric normed space $\mathbb{R}^n$ with the unit ball $B$. Let $1 \le k \le n$.
\begin{enumerate}
  \item The \emph{upper intrinsic inradius} $r^{(k)}_B(K)$ is defined as the largest number $r$ such that, for any codimension $k$ subspace $N$, the homothet $rB$ can be translated into $K+N$.
  \item The \emph{lower intrinsic inradius} $r_{(k)}^B(K)$ is defined as the largest number $r$ such that any $k$-dimensional section (passing through the origin) of $rB$ can be translated into $K$.
\end{enumerate}
\end{definition}

\begin{definition}
\label{def:normedpants}
Let $V = \{v^{1}, \ldots, v^{m}\}$, $m \ge 2$, be a set of points in $\mathbb{R}^n$, such that $\|v^j\| \le r$ for all $j$, and $V$ cannot be covered by a homothet of $B$ smaller than $rB$. Define the \emph{anti-Vorono\u{\i} cells} as
\[
A_{-V}^{j} = \left\{x \in \mathbb{R}^n ~\middle\vert~ \|x+v^{j}\| \ge \|x+v^{j'}\| ~\forall j' \in [m]\right\},
\]
and the \emph{open centered multi-plank} generated by $V$ as
\[
P = \mathbb{R}^n \setminus \bigcup\limits_{j \in [m]} \left(v^{j} + A_{-V}^{j} \right).
\]

The number $r = r^B(P)$ is called the \emph{inradius} of $P$, and the dimension of the convex hull of $V$ is called the \emph{rank} of $P$.
\end{definition}

We remark that the cells $A_{-V}^{j}$ are no longer convex. See Figure~\ref{fig:normed} for an example of a rank 1 plank in an asymmetric norm. In this figure, the unit ball of the norm is depicted in the middle (the origin is marked with a `+' sign), the generating set is $V = \{v^1, v^2\}$.

\begin{figure}[ht]
\centering
\includegraphics[width=0.85\textwidth]{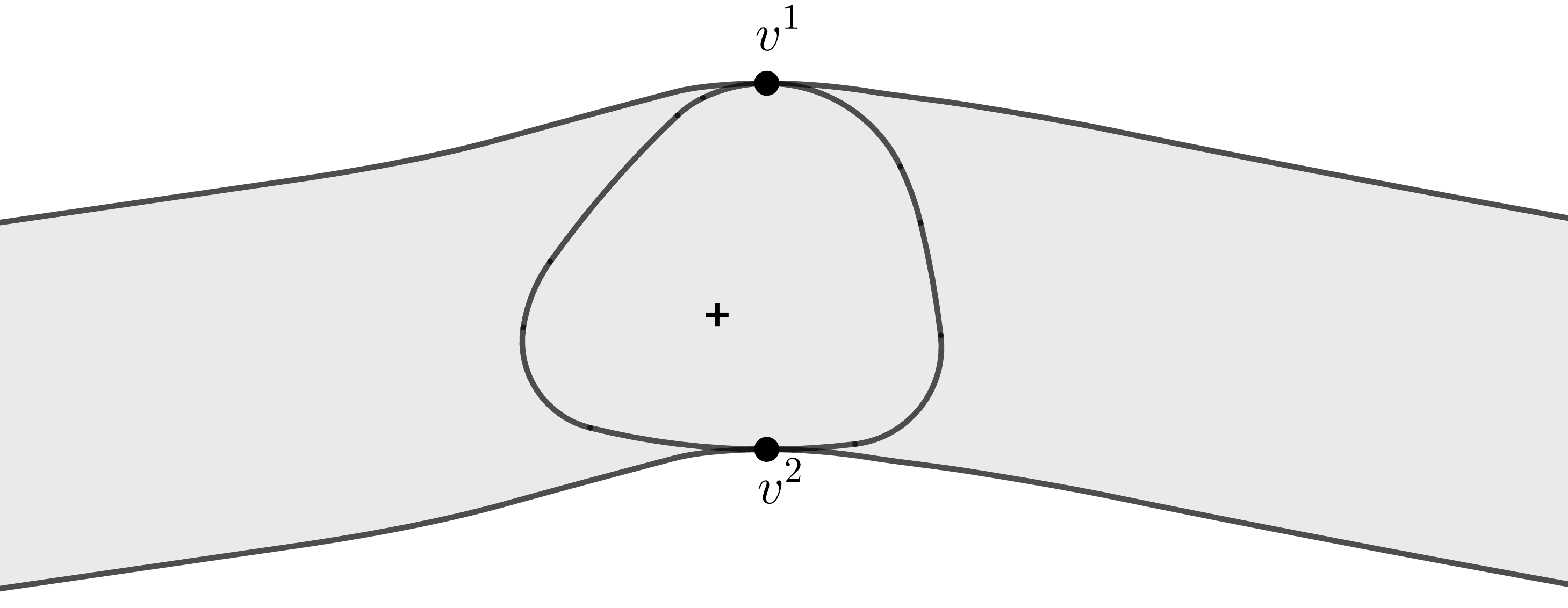}
\caption{A multi-plank in a normed plane}
\label{fig:normed}
\end{figure}

The proof of step 1 in Theorem~\ref{thm:pants} falls through in the normed case. It is no longer true that a shifted multi-plank looks similar to the section of a higher-dimensional multi-plank; this is the reason why we have to consider only centered multi-planks. The proof of step 2 is still valid, though, which gives us the following result.

\begin{theorem}
\label{thm:normedpants}
If a convex set $K$ in an asymmetric normed space $\mathbb{R}^n$ with the unit ball $B$ is covered by (centered) multi-planks $P_1, \ldots, P_N$ of rank at most $k$, then
\[
\sum\limits_{i=1}^N r^B(P_i) \ge r^B_{(k)}(K).
\]
\end{theorem}


In the case $k=1$, $K=B$, Theorem~\ref{def:normedpants} can be viewed as a result on the subadditivity of relative widths. Take a look at the ``bent'' plank $P$ in Figure~\ref{fig:normed}: it has the same length intersection with every line parallel to $v^1 - v^2$. In this sense, $P$ has a well-defined ``relative width'' $r^B(P)$ in this direction. In these terms, Theorem~\ref{def:normedpants} says that if $K$ covered by ``bent'' centered planks then the sum of their ``relative widths'' is at least 1. This might be reminiscent of Bang's conjectured inequality on the sum of relative widths: if an open bounded convex set $K$ containing the origin is covered by (conventional straight) planks $P_1, \ldots, P_N$, then
\[
\sum\limits_{i=1}^N r^{(1)}_K(P_i) \ge 1.
\]

We finish with a strong conjecture subsuming Bang's conjecture as well as many other subadditivity statements.

\begin{conjecture}
Let $B$ be an open bounded convex set containing the origin, and let a convex set $K$ be covered by convex sets $C_1, \ldots, C_N$. Then for any $1 \le k \le n$,
\[
\sum\limits_{i=1}^N r^{(k)}_B(C_i) \ge r_{(k)}^B(K).
\]
\end{conjecture}

\bibliography{kadets}
\bibliographystyle{abbrv}
\end{document}